\documentclass[12pt,a4paper]{article}
\usepackage[utf8x]{inputenc}
\usepackage{amsmath}
\usepackage{amsfonts}
\usepackage{amssymb}
\usepackage{amsthm}
\usepackage{mathrsfs}
\usepackage{fancyhdr}
\usepackage{graphicx}
\usepackage[usenames,x11names]{xcolor}
\usepackage{arabtex}
\usepackage{framed}
\usepackage[top=2cm,bottom=2cm,margin=2cm]{geometry}
\usepackage{cancel}
\usepackage{array}   

\usepackage[dvipdfm,pdfstartview=FitH,colorlinks=true,linkcolor=Red4,urlcolor=Red4,%
filecolor=green,citecolor=red]{hyperref}

\title{A long memory time series with a periodic degree of fractional differencing}
\author{\sc Amine AMIMOUR$^{1}$ \\
and\\	
Karima BELAIDE$^{2}$ \\
Department of Mathematics\\ Applied Mathematics Laboratory\\
University of Bejaia, Bejaia. Algeria \\[1mm]
\href{mailto:amineamimour@gmail.com}{amineamimour@gmail.com}$^{1}$ \\[1mm]
\href{mailto:k\_tim2002@yahoo.fr}{k\_tim2002@yahoo.fr}$^{2}$ \\[1mm]
}
\date{}

\def\EMdash{\leavevmode\hbox to 10.6mm{\vrule height .63ex depth -.59ex
		width 10mm\hfill}}

\theoremstyle{plain}
\numberwithin{equation}{section}
\newtheorem{thm}{Theorem}[section]

\newtheorem{proposition}[thm]{Proposition}
\newtheorem{remark}{Remark}
\begin{document}
	\maketitle
	\begin{abstract}
	This article develops a periodic version of a time varying parameter fractional process in the stationary region. It is a partial extension of Hosking (1981)'s article which dealt with the case where the coefficients are invariant in time. We will describe the probabilistic theories of this periodic model. The results are followed by a graphical representation of the autocovariances functions.
	\end{abstract}
	Simulation. Periodically stationary model. Periodic autocorrelation. Long
	memory. 

\section{Introduction and Notation}
Non-stationary periodically correlated processes find their interesting application in many fields such as economics, hydrology as well as environmental studies. Much attention has been given to the periodic autoregressive moving average process. Most of the extant periodic models in the literature are concerned with the identification, estimation and testing problems. The periodic autoregressive moving average model is introduced by Hannan \cite{H1955} to describe rainfall data. Gladyshev \cite{G1961, G1963} proposed a periodically correlated process and gave a necessary and sufficient condition for a series to be periodic. A discrete time process is a periodically correlated process if there is an integer $p>0$ such that,
\begin{center}
	\begin{equation}
	\label{Eq1}
	E(X_{t+p})\equiv E(X_{t})\text{ and }Cov(X_{t+p\text{ }},X_{t^{^{\prime}}+p})=Cov(X_{t},X_{t^{^{\prime }}}).%
	\end{equation}
\end{center}
The properties of a periodic autoregressive model are stated by Troutman \cite{T1979} adopting the associated stationary multivariate autoregressive model. On the other hand, the invertibility property of the periodic moving average models with its elementary concern has been studied by Cipra \cite{C1985}, Ghysels and Hall \cite{G1992}  and Bentarzi and Hallin \cite{BH1994}. Lund et al \cite{LSB2006} studied the techniques for fitting parsimonious periodic time series models to periodic data, considering the periodic autoregressive moving average time series model, see also Box et al \cite{BJR2008} for more background, all of the articles cited above have taken only a short-memory case. From their description by Lawrance and Kottegoda \cite{LK1977}, considerable attention was given to the long memory time series models in the literature, the introduction of this family model is developed by Hosking \cite{H1981}. He generalized the well-known ARIMA(p,d,q) models of Box and Jenkins to the autoregressive fractionally integrated moving average ARFIMA(p,d,q), this generalization consists of permitting the degree of differencing d to take any real value, specifically $0<d<\frac{1}{2}$, this fractionally
differenced process became a standard model to study the behavior of long term persistence processes and is commonly detected in the analysis of real-life time series data in many areas; for example, in finance, in hydrology, meteorology. He also derived explicit expressions for autocovariances $\gamma_{h}^{X}$,  and
autocorrelation $\rho _{h}^{X}=\frac{\gamma _{h}^{X}}{\gamma _{0}^{X}}$ for
the zero mean purely AFRIMA $\left( 0,d,0\right) $ process $\left(
X_{t},t\in 
\mathbb{Z}
\right) $ represented by $\left( 1-B\right) ^{d}X_{t}=\varepsilon _{t}$, and showed that the autocovariance function $\gamma _{h}^{X}$ satisfies 
\begin{center}
	\begin{equation}
	\label{eq1.1}
	\gamma _{h}^{X}\sim Ch^{-\alpha },\text{ }C>0,\text{ }0<\alpha <1,\text{ }si%
	\text{ }h\rightarrow \infty.
	\end{equation}
\end{center}
An extension of this process in the sense of the periodic variation in the differencing parameter d is proposed in this work, for a know period $p$, the PtvARFIMA idea is to drive an ARFIMA equation at a periodically time varying long memory parameter, this model is proposed by Amimour and Belaide\cite{AB2020}, they proved that it enjoys a local asymptotic normality property. In this article, we will provide some results for this new model. Stressing that this model is a generalization of the model proposed by Hosking \cite{H1981} and can be exploited for expressing better the data which present both the periodic structure and the long term dependence patterns. It is more significative also to consider the short memory exponents in the model of the kind. However, Hui and Li \cite{HL1995} proposed a periodically correlated 2-process composed of two parameters of fractional integration based in two white noise independents,
\begin{center}
	\begin{equation}
	\label{Eq3}
	X_{t}^{(1)}=(1-B)^{-d_{1}}\varepsilon _{t}^{(1)}\text{ and }%
	X_{t}^{(2)}=(1-B)^{-d_{2}}\varepsilon _{t}^{(2)},%
	\end{equation}
\end{center}
where $d_{t}$ is the 2-periodic fractional parameter, they are used this model for modeling the Hong Kong United Christian Hospital attendance series, the data concern seventy five weeks (approximately one and half years) on the average number of people entering the emergency unit observed on the weekday and the weekend, mentioning that the ARMA model is not considered because the autocorrelations indicate a long term dependence pattern in this periodic series. The conditional log likelihood function is determined for estimating the periodic memory parameter. This long memory model is a member of the non-stationary class of process (but periodically stationary). \\
The outline of our article is as follows. We describe the periodic fractional model and these characteristics in the second section. In the section three we calculate the periodic autocovariance and autocorrelation function and we study its asymptotic behavior. The section four contains the graphical representations for the autocovarainces functions.

\section{Definitions and notations}
\label{Sec2}
\subsection{Periodically time-varying long memory parameter}
A purely fractionally differenced process  $(X_{t},t\in 
\mathbb{Z}
)$ with period $p$ $($denoted by PtvARFIMA$_{p}$ $(0,d_{t},0))$ has the
following stochastic equation
\begin{center}
	\begin{equation}
	\label{Eq4}
	(1-B)^{d_{i}}X_{i+pm}=\varepsilon
	_{i+pm},%
	\end{equation}
\end{center}
where for all $t\in 
\mathbb{Z}
$, $\exists $ $i=\left\{ 1,........,p\right\} $, $m\in 
\mathbb{Z}
,$ such that $t=i+pm$ and $p$ represents the period $\in 
\mathbb{N}
$, $d_{t}$ is the periodic degree of fractional differencing whose values lie
in $(0,\frac{1}{2})$, and $\left( \varepsilon _{t},t\in 
\mathbb{Z}
\right) $ is a zero mean white noise with finite variance $\sigma
_{t }^{2}$, the variance is periodic in $t$ such that $\sigma^{2}_{t+pm}$ =$\sigma^{2}_{t}$.

If $d_{i}>0$, the process (\ref{Eq4}) is invertible and has an infinite
autoregressive representation is as follows

\begin{center}
	\begin{equation}
	\label{Eq5}
	\varepsilon _{i+pm}=(1-B)^{d_{i}}X_{i+pm}=\underset{}{\overset{}{\overset{%
				\infty }{\underset{j=0}{\sum }}\text{ }\pi _{j}^{i}X_{i+pm-j}}},%
	\end{equation}
\end{center}
where 
\begin{center}
	$\pi _{j}^{i}=\frac{\Gamma (j-d_{i})}{\Gamma (j+1)\Gamma (-d_{i})}.$
\end{center}

If $d_{i}<\frac{1}{2}$, the process (\ref{Eq4}) is causal and has an infinite moving-average representation is as follows

\begin{center}
	\begin{equation}
	\label{Eq6}
	X_{i+pm}=(1-B)^{-d_{i}}\varepsilon _{i+pm}=\overset{\infty }{\underset{j=0}{%
			\sum }}\psi _{j}^{i}\varepsilon _{i+pm-j},%
	\end{equation}
\end{center}
where
\begin{center}
	$\psi _{j}^{i}=\frac{\Gamma (j+d_{i})}{\Gamma (j+1)\Gamma (d_{i})}.$
\end{center}
\begin{equation}
\label{eq1.4}
\psi _{j}^{i}\sim v_{i}j^{d_{i}-1}, v_{i}>0, as\ j\rightarrow \infty .
\end{equation}

\begin{proposition}
	
	\label{pro2.11}
	Let $(X_{t})_{t\in 
		\mathbb{Z}
	}$ a PtvARFIMA$(0,d_{t},0)$ where, for all $t\in 
	\mathbb{Z}
	,$ there exists $m\in 
	\mathbb{Z}
	$ , $p\in 
	\mathbb{N}
	^{\ast }$; $p>1$ is the period and $%
	d_{i+pm}=d_{i}$ with $i=1,....,p,$ such that $t=i+pm$. The PtvARFIMA$(0,d_{i},0)$ is
	causal for all $i=1,....,p,$ then the series 
	
	\begin{equation}
	\label{Eq2.6}
	\underset{j=0}{\overset{\infty }{\sum }}\psi _{j}^{i}B^{j}(\varepsilon
	_{i+pm})=\underset{j=0}{\overset{\infty }{\sum }}\psi _{j}^{i}\varepsilon
	_{i+pm-j},
	\end{equation}
	converge in quadratic mean.
	
\end{proposition}

\begin{proof}
	
	For all $m\in 
	\mathbb{Z}
	$ , $i=1,.....,p,$ the coefficients $\psi _{j}^{i}$ are given in (\ref{Eq6}), let $t,s$ a positive integers, such that $t<s,$
	we define $S_{t}=\overset{t}{\underset{j=0}{\sum }}$ $\psi
	_{j}^{i}\varepsilon _{i+pm-j},$
	then for $i=1,.....,p.$ we have
	
	\begin{eqnarray*}
		||S_{s}-S_{t}||^{2} &=&E\left[ \left\vert \overset{s}{\underset{v=0}{\sum }}\psi
		_{v}^{i}\varepsilon _{i+pm-v}-\overset{t}{\underset{j=0}{\sum }}\psi
		_{j}^{i}\varepsilon _{i+pm-j}\right\vert ^{2}\right]  \\
		&=&E\left[ \overset{s}{\underset{j=t+1}{\sum }}(\psi
		_{v}^{i})^{2}\varepsilon _{i+pm-v}^{2}+\overset{t}{\underset{\underset{v\neq
					j}{v,j=t+1}}{\sum }}\psi _{v}^{i}\psi _{j}^{i}\varepsilon
		_{i+pm-v}\varepsilon _{i+pm-j}\right]  \\
		&=&\sigma _{i}^{2}\underset{j=t+1}{\overset{s}{\sum }}(\psi _{j}^{i})^{2}.
	\end{eqnarray*}	
	By the Cauchy criterion, one can easily verify that
	$\underset{j=t+1}{\overset{s}{\sum }}(\psi
	_{j}^{i})^{2}<\infty ,$ for all $i=1,...,p.$
\end{proof}
\begin{proposition}
	\label{pro2.12}
	Let $(X_{t})_{t\in 
		\mathbb{Z}
	}$ a PtvARFIMA$(0,d_{t},0)$ where, for all $t\in 
	\mathbb{Z}
	,$ there exists $m\in 
	\mathbb{Z}
	$ , $p\in 
	\mathbb{N}
	^{\ast }$; $p>1$ is the period and $%
	d_{i+pm}=d_{i}$ with $i=1,....,p,$ such that $t=i+pm$. The PtvARFIMA$(0,d_{i},0)$ is
	invertible for all $i=1,....,p,$ then the series 
	
	\begin{equation}
	\underset{j=0}{\overset{\infty }{\sum }}\pi _{j}^{i}B^{j}(X
	_{i+pm})=\underset{j=0}{\overset{\infty }{\sum }}\pi _{j}^{i}X
	_{i+pm-j},
	\end{equation}
	converge in quadratic mean.
\end{proposition}

\begin{proof}

	The proof is similar to the first proof of the proposition (\ref{pro2.11}).
\end{proof}
\subsection{Periodically correlated process}

\begin{proposition}
	\label{prop1}
	Let $(X_{t})_{t\in 
		\mathbb{Z}
	}$ a PtvARFIMA$(0,d_{t},0)$ where, for all $t\in 
	\mathbb{Z}
	,$ there exists $m\in 
	\mathbb{Z}
	$ , $p\in 
	\mathbb{N}
	^{\ast }$; $p>1$ is the period and $%
	d_{i+pm}=d_{i}$ with $i=1,....,p,$ such that $t=i+pm$. The PtvARFIMA$(0,d_{i},0)$ is a periodically correlated process.
\end{proposition}
\begin{proof}
	From the equation (\ref{Eq6}) the proof can be obtained fairly simply, it is easy to
	show that $\psi _{j}^{t+pm}=\psi _{j}^{t}$, is periodic with period $p,$ then
	
	\begin{eqnarray*}
		\gamma _{X}(t+pm,t{^{\prime}}+pm) 
		&=&E(X_{t+pm} X_{t{^{\prime}}+pm})\\
		&=&E(\underset{j=0}{\overset{\infty }{\sum }}\psi _{j}^{t+pm}\varepsilon
		_{t+pm-j} \overset{\infty }{\underset{j=0}{\sum }}\psi
		_{j}^{{t{^{\prime}}+pm}}\varepsilon _{{t{^{\prime}}+pm}-j}) \\
		&=&\underset{j=0}{\overset{\infty }{\sum }}\psi _{j}^{t+pm} \psi_{j}^{{t{^{\prime}}+pm}}E(\varepsilon _{t}^2)\\
		&+&\underset{j \neq i}{\overset{\infty }{\sum }}\psi _{j}^{t+pm} \psi_{j}^{{t{^{\prime}}+pm}}E(\varepsilon _{t-j}\varepsilon
		_{t-i})\\
		&=&\sigma ^{2}\underset{j=0}{\overset{\infty }{\sum }}\psi _{j}^{t} \psi_{j}^{{t{^{\prime}}}}\\
		&=&\gamma _{X}(t,t{^{\prime }}).
	\end{eqnarray*}
\end{proof}

\section{Periodic autocovariance function}
\label{Sec3}

\subsection{Periodic autocovariance function with period two}
\begin{proposition}
	\label{prop2}
	For $p=2$ then $i=1$ or $2$, the autocovariance function of the model \ref{Eq4} is given by
	
	\begin{equation}
	\label{Eq7}
	\gamma _{X}^{i}(h)=\left\{ 
	\begin{array}{c}
	\sigma _{i }^{2}\frac{\Gamma (1-d_{i}-d_{i+1})\Gamma (d_{i+1}+h)}{%
		\Gamma (d_{i+1})\Gamma (1-d_{i+1})\Gamma (1+h-d_{i})}\text{ if h odd} \\ 
	\sigma _{i }^{2}\frac{\Gamma (1-2d_{i})\Gamma (d_{i}+h)}{\Gamma
		(d_{i})\Gamma (1-d_{i})\Gamma (1+h-d_{i})}\text{ if h even}%
	\end{array}%
	\right. ,%
	\end{equation}
\end{proposition}	
\begin{proof}
	When $i=1$ or $2$, the autocovariance function can be calculated by%
	\begin{equation*}
	\gamma _{X}^{i}(h)=\sigma _{i }^{2}\underset{j=0}{\overset{\infty }%
		{\sum }}\frac{\Gamma (j+d_{i})}{\Gamma (d_{i})\Gamma (j+1)}\frac{\Gamma
		(j+d_{i+h}+h)}{\Gamma (d_{i+h})\Gamma (j+1+h)},
	\end{equation*}
	
	$\gamma _{X}^{i}(h)$ depends to $i$ the model $X_{i+2m\text{ \ }}$is not
	stationary but periodically stationary.(proposition \ref{prop1}),
	by a simpl calculus the autocovariance function may be written as 
	\begin{equation}
	\label{Eq8}
	\gamma _{X}^{i}(h)=\sigma _{i }^{2}\frac{\Gamma (d_{i+h}+h)}{\Gamma (d_{i+h})\Gamma (h+1)}%
	F(d_{i},d_{i+h}+h,1+h,1),%
	\end{equation}
	where F is the hypergeometric function, see Abramowitz and Stegun (\cite{AS1965}, 556),
	next and using the propertie of hypergeometric function, we obtain
	
	\begin{equation}
	\label{Eq9}
	\gamma _{X}^{i}(h)=\sigma _{i }^{2}\frac{\Gamma
		(1-d_{i}-d_{i+h})\Gamma (d_{i+h}+h)}{\Gamma (d_{i+h})\Gamma
		(1-d_{i+h})\Gamma (1+h-d_{i})},%
	\end{equation}
\end{proof}

\begin{remark}
	\label{Rem1}
	We will generalize this result later in the case
	where the period $p$\emph{\ is in }$%
	\mathbb{N}
	.$
\end{remark}

\subsubsection{Asymptotic behavior of autocovariances with period two}

\begin{proposition}
	\label{prop4}
	\begin{equation}
	\label{Eq10}
	\gamma _{X}^{i}(h)\simeq \left\{ 
	\begin{array}{c}
	\sigma _{i }^{2}  \frac{\Gamma (1-d_{i}-d_{i+1})}{\Gamma
		(d_{i+1})\Gamma (1-d_{i+1})}(h)^{d_{i}+d_{i+1}-1}	\ \ if \text{ }h\text{ odd } \\ 
	\sigma _{i }^{2} \frac{\Gamma (1-2d_{i})}{\Gamma (d_{i})\Gamma
		(1-d_{i})}(h)^{2d_{i}-1}\ if \text{ }h\text{ even}%
	\end{array}%
	\right.%
	\end{equation}
	For h large, the components of the autocovariance function ($\gamma
	_{X}^{1}(h),\gamma _{X}^{2}(h))$ decay at a hyperbolic rate.(slowly
	decreases to 0 when h tends to infinity).
\end{proposition}

\begin{proof}
	According to (\ref{Eq9}) and using the standard approximation derived from
	Stirling's formula, that for h large, $\frac{\Gamma (d_{i+h}+h)}{\Gamma
		(1+h-d_{i})}$ is well approximated by $(h)^{d_{i}+d_{i+h}-1}.$ 
\end{proof}

\subsection{Periodic autocovariance function with period p$\in 
	\mathbb{N}
	; p>1$}

Now we will generalize the results that we have facilitated in proposition ~\ref{prop2} when $p\in 
\mathbb{N}
$ $i.e$ $i=1,.......,p$ and $h\equiv k[p]$ , we take the formula~\ref{Eq9}, we have $p$ components 
\begin{center}
	\begin{equation}
	\label{Eq11}
	\gamma _{X}^{i}(h)=\sigma _{i }^{2}\frac{\Gamma
		(1-d_{i}-d_{i+h})\Gamma (d_{i+h}+h)}{\Gamma (d_{i+h})\Gamma
		(1-d_{i+h})\Gamma (1+h-d_{i})},%
	\end{equation}
	
\end{center}
using the standard approximation derived from Stirling's formula, that for
large $h$. It follows that the approximation of $\gamma _{X}^{i}(h)$ is

\begin{center}
	\begin{equation}
	\label{Eq12}
	\gamma _{X}^{i}(h)\simeq \sigma _{i }^{2}\frac{\Gamma
		(1-d_{i}-d_{i+k})}{\Gamma (d_{i+k})\Gamma (1-d_{i+k})}(h)^{d_{i}+d_{i+k}-1}.%
	\end{equation}

\end{center}
Then, 
\begin{center}
	$\gamma_{X}^{i}(h) \sim C_{i}h^{-\alpha_{i}},\ C_{i}>0,\ 0<\alpha_{i} <1,\ as\ h\rightarrow \infty,$
\end{center}
with
\begin{center}
	$\alpha_{i} =-d_{i}-d_{i+k}+1$, $h\equiv k[p],$
\end{center}

and

\begin{equation}
\label{eq1.2}
C_{i}=\sigma_{i} ^{2}\frac{\Gamma
	(1-d_{i}-d_{i+k})}{\Gamma (d_{i+k})\Gamma (1-d_{i+k})}%
.
\end{equation}
\section{Autocorrelation function for a PtvARFIMA model}

\begin{proposition}
	\label{prop5}
	The periodic autocorrelation function with period p$\in 
	\mathbb{N}
	$ (i.e $i=1,.......,p$ and $h\equiv k[p]$) is given by 
	\begin{equation}
	\label{Eq15}
	\rho _{X}^{i}(h)=\frac{[\Gamma (1-d_{i})]^{2}\Gamma (1-d_{i}-d_{i+k})\Gamma
		(d_{i+k}+h)}{\Gamma (1-2d_{i})\Gamma (d_{i+k})\Gamma (1-d_{i+k})\Gamma
		(1+h-d_{i})}.%
	\end{equation}

	\begin{proof}
		
		For p=2, i=1 the autocorrelation function can be calculated by 
		\begin{equation*}
		\rho _{X}^{1}(h)=\frac{\gamma _{X}^{1}(h)}{\gamma _{X}^{1}(0)}=\left\{ 
		\begin{array}{c}
		\frac{\lbrack \Gamma (1-d_{1})]^{2}\Gamma (1-d_{1}-d_{2})\Gamma (d_{2}+h)}{%
			\Gamma (d_{2})\Gamma (1-d_{2})\Gamma (1-2d_{1})\Gamma (1+h-d_{1})}	\text{ if h odd} \\ 
		\frac{\Gamma (1-d_{1})\Gamma (d_{1}+h)}{\Gamma (d_{1})\Gamma (1+h-d_{1})}\text{ if
			h even}%
		\end{array}%
		\right.,
		\end{equation*}
		similary for i =2, then we can deduce the (proposition~\ref{prop5})
	\end{proof}
	
\end{proposition}		
From the structure of autocovariance or autocorrelation  obtained, we deduce that the PtvARFIMA model is a persistent process, more precisely, the periodic autocorrelation function  $\rho _{X}^{i}(h)$ does not decrease at a geometric rate, but exhibits an asymptotic behavior equivalent to $(h)^{d_{i}+d_{i+k}-1}$, with $h\equiv k[p]$ as $h$ tends to infinity, it means that these correlations are not absolutely summable. In addition, the autocorrelations are expressed as a function of the different long memory parameters, distributed periodically according to the lag $h$.

	\section{The graphical representations}
	
In this section we represent graphically, the periodic autocovariance function for a period $p=2$, with different values of 
	$d_{i}$, for us we show the influence of the memory parameter $d$ on the behavior of the covariance function. We will take two cases, in the first case we take the two different values of d$_{i}$ : $%
	d_{1}=0.3,d_{2}=0.4,$ in the second we take d: $d_{1}=0.09,d_{2}=0.49.$
	\begin{figure}[h!]
		
		\centering
		\includegraphics[width=11cm]{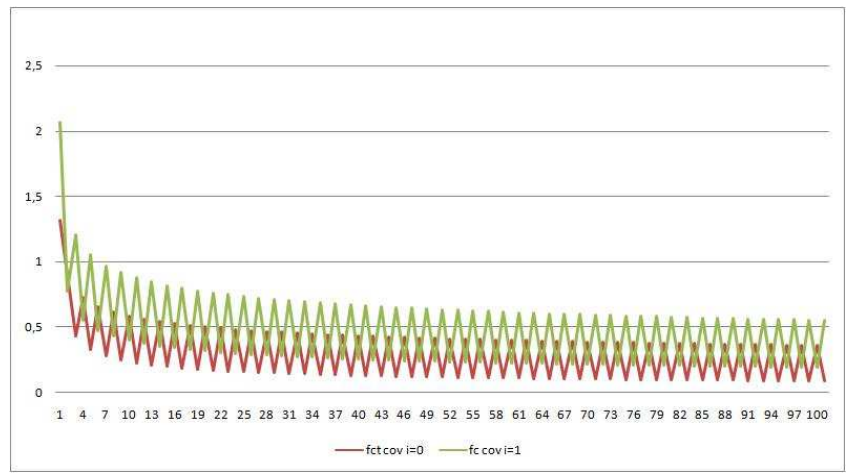}
		\caption{Periodic autocovariance function,
			with period 2 and 	lag $\ h = 0$ to$\ 100,\ d_{1}=0.3,\ d_{2}=0.4.$}	\label{Fig1}
	\end{figure}
	\vspace*{5mm}
	
	The figure(Fig~\ref{Fig1}) , illustrates that the periodicity is caused by the fractional
	parameters $d=(0,3;0,4)$, we see that, $\gamma _{X}^{1}(h)<\gamma
	_{X}^{2}(h) $ because the value of $d_{1}$ is less than the value of $d_{2}$%
	, the periodic autocovariance $\gamma _{X}^{i}(h)$ decrease hyperbolically,
	for the values of $d_{i}$ close to 0, we find that the autocovariance
	function decreases rapidly towards zero.
	
	\begin{figure}[h!]
		
		\centering
		
		\includegraphics[width=11cm]{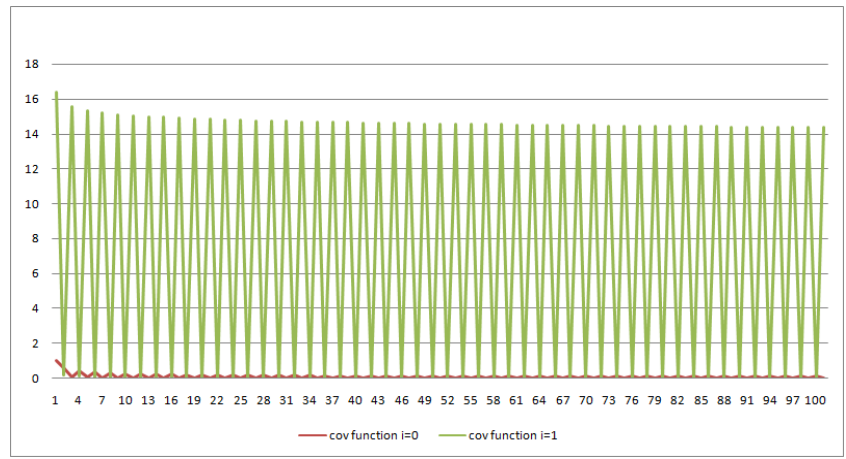}
		\caption{Periodic autocovariance function,
			with period 2 and 	lag $\ h = 0$ to$\ 100,\ d_{1}=0.09,\ d_{2}=0.49.$
		}\label{Fig2}
	\end{figure}
	\vspace*{5mm}
	
	The figure(Fig~\ref{Fig2}) confirms the previous observation because for a value of d
	close to $\frac{1}{2}$ we notice that the autocovariance function decreases
	very slowly, whereas for the value of d very close to 0 we observe a strong
	decay towards 0 for the function of covariance. Between the
	two figures(Fig~\ref{Fig1}) and(Fig~\ref{Fig2}), in the second case a very important shift is observed
	between the functions in figure(Fig~\ref{Fig2})  for the values of d distant ($%
	d_{1}=0.09,d_{2}=0.49$ ), while for the values of d close enough ($%
	d_{1}=0.3,d_{2}=0.4$ ), there is less shift between the functions in Figure(Fig~\ref{Fig1})
	(first case), which shows the influence of the parameters of
	differentiation on the values of the function as well on their behavior. Indicates a long term dependence and the presence of a cycle of different periodicity.
	
\section{Conclusion}
\label{sec:conc}	
The article presents a model inspired by Hosking $\cite{H1981}$, considering the long memory parameter, d, a periodically time-varying function. We have described the asymptotic proprties of the model, and studied the behavior of the autocovariance function, which was illustrated by the graphical representations. The results could be the opening wedge of the new subjects in the periodic ARFIMA models.


\begin{thebibliography}{99}
		\bibitem{AB2020}
		{\sc A. Amimour,  K. Belaide}, Local asymptotic normality for a periodically time varying long memory parameter. {\it Communications in Statistics - Theory and Methods, doi : 10.1080/03610926.2020.1784435}, (2020).		
		\bibitem{AS1965}
		{\sc M. Abramowitz, I. Stegun}, Handbook of Matematical Functions. {\it New York: Dover Publications }, (1965).	
		\bibitem{BH1994}
		{\sc M. Bentarzi,  M. Hallin}, On the invertibility of periodic moving-average models.
		{\it Journal of Time Series Analysis}, {\bf{15}} (1994), pages 263--268.			
		\bibitem{BJR2008}
		{\sc G.E.P. Box, G.M. Jenkins, G.C. Reinsel}, Time Series Analysis: Forecasting and Control. {\it  New Jersey: 4th ed. Wiley.}, (2008).
		\bibitem{C1985}
		{\sc T. Cipra}, Periodic moving average processes. {\it  Aplikace Matematily}, {\bf{30}}(1985a), pages 218--229 , .
		\bibitem{G1992}
		{\sc E. Ghysels, A. Hall}, Testing periodicity in some linear macroeconomic models. {\it  Unpublished manuscript, CRDE, Montreal}, (1992).
		\bibitem{G1961} 
		{\sc E.G. Gladyshev}, Periodically correlated random sequences. {\it Soviet. Mathematics} {\bf{2}} (1961), pages 385--388.	
		\bibitem{G1963}
		{\sc E.G. Gladyshev}, Periodically and almost PC random processes with continuous time parameter. {\it Theory Probability and its Applications}, {\bf{8}} (1963), pages 173–-177.
		\bibitem{G1988}
		{\sc E. Goncalves}, "Processus fractionnaires" [Fractionally processes] phD diss. {\it  University of Science and  Technology Lille Flanders Artois}, (1988).	
		\bibitem{H1955}
		{\sc E.J. Hannan}, A test for singularities in Sydney rainfall. {\it  Australian Journal of Physics}, {\bf{8}}(1955), pages 289--97.
		\bibitem{H1981} 
		{\sc J.R.M. Hosking}, Fractional differencing. {\it Biometrika.}, {\bf{68}} (1981), pages 165--176.
		\bibitem{HL1995} 
		{\sc Y.V. Hui,  W.K. Li}, On fractionally differenced periodic processes. {\it The Indian Journal of Statistics}, {\bf{57}} (1995), pages 19--31. 
		\bibitem{LK1977}
		{\sc A.J. Lawrance, N. T. Kottegoda}, Stochastic modelling of river-flow time series. {\it Journal of the Royal Statistical Society}, {\bf{140}} (1977), pages 1--31.
		\bibitem{LSB2006}
		{\sc R. Lund, Q. Shao, I. Basawa}. Parsimonious Periodic Time Series Modeling.
		{\it  Australian and New Zealand Journal of Statistics}. {\bf{48}} (2006), pages 33--47.			
		\bibitem{P1990}
		{\sc S. Porter-Hudak}, An aplication of the seasonal fractionally differenced model to the monetary aggegrates.
		{\it Journal of the American Statistical Association}, {\bf{85}} (1990), pages 338--344.	
		\bibitem{T1979} 
		{\sc B.M. Troutman},Some results in periodic autoregression. 
		{\it Biometrika}, {\bf{2}} (1979), pages 219--28.

\end{thebibliography}
\end{document}